\documentclass[11pt]{amsart}
\usepackage{amscd, amsmath, amsthm, amssymb}
\usepackage{color}
\usepackage{graphicx}
\usepackage{mathrsfs}
\usepackage[all]{xy}
\usepackage{enumerate}
\usepackage{hyperref}
\usepackage{verbatim}

\newtheorem{theorem}{Theorem}[section]
\newtheorem{lemma}[theorem]{Lemma}
\newtheorem{proposition}[theorem]{Proposition}
\newtheorem{conjecture}[theorem]{Conjecture}
\newtheorem{question}[theorem]{Question}
\newtheorem{corollary}[theorem]{Corollary}

\theoremstyle{definition}
\newtheorem{definition}[theorem]{Definition}
\newtheorem{example}[theorem]{Example}

\newtheorem{setup}[theorem]{Setup}
\newtheorem{construction}[theorem]{Construction}
\newtheorem{assumption}[theorem]{Assumption}

\theoremstyle{remark}
\newtheorem{remark}[theorem]{Remark}
\newtheorem{caution}[theorem]{Caution}

\numberwithin{equation}{section}

 \newcommand\CC{{\mathbb{C}}}
\newcommand\C{{\mathcal{C}}}   \newcommand\GG{{\mathbb{G}}}  \newcommand\II{\mathcal{I}}
\newcommand\LL{{\mathbf{L}}} \newcommand\mm{{\mathbf{m}}} \newcommand\NN{{\mathbb{N}}} \newcommand\OO{{\mathscr{O}}} \newcommand\PP{{\mathbb{P}}} \newcommand\QQ{{\mathbb{Q}}} \newcommand\RR{{\mathbb{R}}}   \newcommand\ZZ{{\mathbb{Z}}}

  \newcommand\Aut{{\rm Aut}}  \newcommand\Bir{{\rm Bir}} \newcommand\Eff{{\rm Eff}}
\newcommand\NE{\overline{{\rm NE}}}

\newcommand\Nef{{\rm Nef}}
\newcommand\Nefe{{\rm Nef}^e}
\newcommand\Nefp{{\rm Nef}^{+}}
\newcommand\Movo{{\rm Mov}}
\newcommand\Movox{{\rm Mov}(X)^{\circ}}
\newcommand\Mov{\overline{{\rm Mov}}}
\newcommand\Move{\overline{{\rm Mov}}^e}

\newcommand\Movp{{\rm Mov}^{+}}
\newcommand\Pic{\text{\rm Pic}}
\newcommand\Psef{{\rm Psef}}

\newcommand\Cone{{\rm Cone}} \newcommand\Cox{{\rm Cox}}
 \newcommand\GL{{\rm GL}}  \newcommand\Hom{{\rm Hom}}
   \newcommand\mld{{\rm mld}}
\newcommand\rank{{\rm rank\,}} 
 \newcommand\Spec{{\rm Spec}}

\begin{document}

\title[Remarks on nef and movable cones]{Remarks on nef and movable cones of hypersurfaces in Mori dream spaces}

\author[Long\ \ Wang]{Long\ \ Wang}
\address{Graduate School of Mathematical Sciences, the University of Tokyo, 3-8-1 Komaba, Meguro-Ku, Tokyo 153-8914, Japan}
\email{wangl11@ms.u-tokyo.ac.jp}

\begin{abstract} We investigate nef and movable cones of hypersurfaces in Mori dream spaces. The first result is: Let $Z$ be a smooth Mori dream space of dimension at least four whose extremal contractions are of fiber type of relative dimension at least two and let $X$ be a smooth ample divisor in $Z$, then $X$ is a Mori dream space as well.

The second result is: Let $Z$ be a Fano manifold of dimension at least four whose extremal contractions are of fiber type and let $X$ be a smooth anti-canonical hypersurface in $Z$, which is a smooth Calabi--Yau variety, then the unique minimal model of $X$ up to isomorphism is $X$ itself, and moreover, the movable cone conjecture holds for $X$, namely, there exists a rational polyhedral cone which is a fundamental domain for the action of birational automorphisms on the effective movable cone of $X$. 
%$\mathrm{Bir}(X)$ on $\overline{\mathrm{Mov}}(X)\cap \mathrm{Eff}(X)$. 

The third result is: Let $P:= \PP^n \times \cdots \times \PP^n$ be the $N$-fold self-product of the $n$-dimensional projective space. Let $X$ be a general complete intersection of $n+1$ hypersurfaces of multidegree $(1, \dots, 1)$ in $P$ with $\dim X \geq 3$. Then $X$ has only finitely many minimal models up to isomorphism, and moreover, the movable cone conjecture holds for $X$.
\end{abstract}

%\date{\today}

%\subjclass[2010]{14E30, 14J70}	

%\keywords{Mori dream space, Lefschetz-type theorem, Nef cone}



\maketitle

\thispagestyle{empty}

\setcounter{section}{0}

\section{Introduction}

The notion of a \textit{Mori dream space} (see Definition \ref{hk1.10}) was introduced by Hu and Keel in \cite{HK00}. The name comes from the fact that Mori dream spaces have very nice properties from the viewpoint of birational geometry, namely, the minimal model program can be carried out for any divisor. On the other hand, a Mori dream space is a GIT quotient of an affine variety by a torus in a natural way as Cox's construction of toric varieties as quotients of affine spaces (\cite{Co95}). Basic examples of Mori dream spaces include toric varieties and Fano varieties.

Being a Mori dream space is rather restrictive and is destroyed under various operations in general, for instance, taking ample divisors. In contrast, Jow showed the following result by using geometric invariant theory in \cite[Corollary 2]{Jo11} (see also \cite[Theorem 2.1 and Remark 2.4]{AL12}).

\medskip  \textit{Let $Z$ be a smooth projective variety of dimension $\geq 4$. Suppose $Z$ is a product of some Mori dream spaces, each having dimension $\geq 2$ and Picard number $1$. Then $Z$ is a Mori dream space, and every smooth ample divisor $X \subset Z$ is a Mori dream space such that the restriction map identifies $\Nef(Z)$ with $\Nef(X)$. }

\medskip In this paper, we first give a generalization via a different approach by using birational geometry, which is inspired by the argument of Koll\'{a}r in \cite{Ko91}.

\begin{theorem}\label{main} Let $Z$ be a smooth Mori dream space of dimension $\geq 4$ whose extremal contractions are of fiber type of relative dimension at least $2$. Let $i: X \hookrightarrow Z$ be a smooth ample divisor in $Z$. Then $X$ is a Mori dream space as well,  and moreover,
\[ \Eff(Z) = \Movo(Z) = \Nef(Z) = i^{\ast}(\Nef(X)) = i^{\ast}(\Movo(X)) = i^{\ast}(\Eff(X)). \]
\end{theorem}

Let $X$ be a smooth ample divisor in a Mori dream space $Z$. There are various obstructions to $X$ being a Mori dream space (see Section \ref{prof}). One is that there exist nef divisors on $X$ that are not semi-ample, particularly the nef cones of $X$ and $Z$ do not coincide in this case. Another one is that the movable cone of $X$ fails to be rational polyhedral even if the nef cones of $X$ and $Z$ coincide --- typical examples are Calabi--Yau hypersurfaces in Fano manifolds.

Although a Calabi--Yau variety (see Definition \ref{caly}) is not necessarily a Mori dream space in contrast to Fano varieties, the cone conjecture due to Morrison and Kawamata (Conjecture \ref{coneconj}) and the generalized abundance conjecture (Conjecture \ref{abundconj}) predict that up to the action of (birational) automorphisms, a Calabi--Yau variety is a ``Mori dream space''. These two conjectures are widely open even in dimension three. In the rest of this paper, we concentrate on the former and our second result is the following.

\begin{theorem}\label{main2} Let $Z$ be a Fano manifold of dimension $\geq 4$ whose extremal contractions are of fiber type. Let $X \in |-K_Z|$ be a smooth member, which is a smooth Calabi--Yau variety. Then $X$ admits the unique minimal model itself up to isomorphism $($as abstract varieties$)$. Moreover, the movable cone conjecture holds for $X$, namely, there exists a rational polyhedral cone which is a fundamental domain for the action of $\Bir(X)$ on $\Move(X):= \Mov(X)\cap \Eff(X)$.
\end{theorem}

This covers a recent result from \cite{Ya21}, which is a generalization of previous works \cite{CO15, Og14, Og18, Sk17} (see Example \ref{compl}). One can also apply the above theorem to some new cases including Calabi--Yau hypersurfaces in rational homogeneous manifolds or products of Fano manifolds with Picard number one (Examples \ref{homog} and \ref{prod}). The proof is inspired by the argument of Oguiso in \cite[Section 6]{Og14} (see also \cite{CO15, Og18, Sk17, Ya21}). The key is to show the first claim that the only minimal model of $X$ up to isomorphism is $X$ itself. The second claim then follows from the first one together with some general results on the structure of movable cones presented in Section \ref{gener}, which might be of independent interest.

We note that non-trivial minimal models may appear for certain Calabi--Yau complete intersections in a self-product of projective spaces (see \cite{HT14, HT18}, also Remark \ref{htsec3}). The following is the third result of this paper.

\begin{theorem}\label{cipp} Let $P:= \PP^n \times \cdots \times \PP^n$ be the $N$-fold self-product of the $n$-dimensional projective space. Let $X$ be a general\footnote{See Assumptions \ref{assump} and \ref{assump2}} complete intersection of $n+1$ hypersurfaces of multidegree $(1, \dots, 1)$ in $P$ with $\dim X \geq 3$. Then $X$ has only finitely many minimal models up to isomorphism $($as abstract varieties$)$. Moreover, there exists a rational polyhedral cone which is a fundamental domain for the action of $\Bir(X)$ on $\Move(X)$.
\end{theorem}

Our approach is inspired by \cite{Fr01, HT14, HT18, Li16, LW21}, which in fact traces back to Cayley's work on determinantal quartic surfaces \cite{Ca69} in the nineteenth century. We recommend to the interested reader \cite{FGGL13} for more details about the history. Given $X$ as above, we can explicitly construct its small elementary contractions and the associated flops (see Construction \ref{constr}). The finiteness of minimal models then follows from a result of Kawamata (\cite{Ka08}) and the symmetry of our construction. Once we can show the finiteness of minimal models, the existence of fundamental domain follows similarly.

After some preliminaries in Section \ref{prel}, we prove Theorem \ref{main} in Section \ref{prof}. Some general results towards the cone conjecture are presented in Section \ref{gener}. The proofs of Theorem \ref{main2} and Theorem \ref{cipp} are given in Section \ref{prof2} and Section \ref{prof3}, respectively. In Section \ref{spec}, we will present the construction of minimal models for a special case, which might help readers to understand the general construction in Section \ref{prof3}.

%\medskip

\section{Preliminaries}\label{prel}

We work over the field $\CC$ of complex numbers and refer to \cite{Ma02} for the knowledge of minimal model program.

Let $X$ be a normal $\QQ$-factorial projective variety. We denote by $N^1(X)$ the group of Cartier divisors modulo numerical equivalence, which is a finitely generated abelian group. Recall that an effective divisor is called \textit{movable} if its stable base locus has codimension at least $2$. The cones generated by nef divisors, effective divisors, and movable divisors in $N^1(X)_{\RR} $ are denoted by $\Nef(X)$, $\Eff(X)$, and $\Movo(X)$, respectively. Note that the last two cones are neither open nor closed in general. We denote by $\Mov(X)$ the closure and by $\Movox$ the interior of $\Movo(X)$. We define the \textit{effective nef cone} as
\[ \Nefe(X) := \Nef(X)\cap \Eff(X) \]
and the \textit{effective movable cone} as
\[ \Move(X) := \Mov(X)\cap \mathrm{Eff}(X). \]
Similarly, we denote by $N_1(X)$ the group of $1$-cycles modulo numerical equivalence, and by $\NE(X) \subset N_1(X)_{\RR}$ the closure of the cone of effective $1$-cycles, which is the dual of $\Nef(X)$ under the intersection pairing. A \textit{small $\QQ$-factorial modification} of $X$ is a birational map $X \dashrightarrow Y$ to another normal $\QQ$-factorial projective variety $Y$ which is isomorphic in codimension $1$.

\begin{definition}[{\cite[Definition 1.10]{HK00}}]\label{hk1.10} Let $X$ be a normal $\QQ$-factorial projective variety. We call $X$ a \textit{Mori dream space}, if the following three conditions are satisfied.

\medskip $(1)$ $\Pic(X)_{\QQ} = N^1(X)_{\QQ}$, or equivalently, $h^1(X, \OO_X) = 0$.

\medskip $(2)$ $\Nef(X)$ is generated by finitely many semi-ample divisors as a convex cone.

\medskip $(3)$ There are finitely many small $\QQ$-factorial modifications $f_i: X \dashrightarrow X_i$, $1 \leq i \leq n$, such that each $X_i$ satisfies $(1)$, $(2)$, and
\[ \Movo(X) = \bigcup^n_{i = 1}f^{\ast}_i(\Nef(X_i)). \]
\end{definition}

We refer to \cite{HK00, McK10, Ok16, Ca18} for more details about Mori dream spaces and related topics. A birational map $f: X \dashrightarrow Y$ between normal projective varieties is called \textit{contracting} if the inverse map $f^{-1}$ does not contract any divisors.

\begin{proposition}[{\cite[Proposition 1.11]{HK00}}]\label{hk1.11} Let $X$ be a Mori dream space. Then the following hold.

\medskip $(\mathrm{1})$ Minimal model program can be carried out for any divisor on $X$. That is, the necessary contractions and flips exist, any sequence terminates, and if at some point the divisor becomes nef then at that point it becomes semi-ample.

\medskip $(\mathrm{2})$ There are finitely many birational contractions $g_i : X \dashrightarrow Y_i$, with $Y_i$ a Mori dream space, such that
\[ \Eff(X) = \bigcup_i \C_i, \]
\[ \C_i = g_i^{\ast}(\Nef(Y_i)) + \RR_{\geq 0}\cdot E_{i_1} + \cdots + \RR_{\geq 0}\cdot E_{i_k}, \]
where $E_{i_1}, \dots, E_{i_k}$ are the prime divisors contracted by $g_i$. The cones $\C_i$ are called the Mori chambers of $X$.
\end{proposition}

For a Mori dream space $X$, an \textit{extremal ray} of $\NE(X)$ is a $1$-dimensional face of $\NE(X)$; a \textit{extremal contraction} $X$ is a contraction associated with some extremal face (i.e., a face of $\NE(X)$ spanned by extremal rays).

%We notice the following easy fact which is implied by Definition \ref{hk1.10} and Proposition \ref{hk1.11} immediately.

%\begin{lemma} Let $X$ be a Mori dream space. Then $\Nef(X) = \Mov(X) = \Eff(X)$ if and only if each extremal contraction of $X$ is of fiber type.
%\end{lemma}

\begin{definition} Let $X$ be a projective variety satisfying that $\Pic(X)_{\QQ} = N^1(X)_{\QQ}$. By a \textit{Cox ring} for $X$ we mean the ring
\[ \Cox(X, \LL) := \bigoplus_{\mm \in \ZZ^{\rho}}H^0(X, \LL^{\mm}), \]
where $\LL = (L_1, \dots, L_{\rho})$ are line bundles which form a basis of $\Pic(X)_{\QQ}$, $\mm = (m_1, \dots, m_{\rho})$, and $\LL^{\mm} = L_1^{\otimes m_1}\otimes \cdots\otimes L_{\rho}^{m_{\rho}}$.
\end{definition}

Although the definition of $\Cox(X, \LL)$ depends on a choice of basis $\LL$, whether or not it is finitely generated is independent of this choice. See \cite{LV09, ADHL14} for more details on Cox rings.

\begin{proposition}[{\cite[Proposition 2.9]{HK00}}]\label{hk2.9} Let $X$ be a normal $\QQ$-factorial projective variety satisfying $\Pic(X)_{\QQ} = N^1(X)_{\QQ}$. Then the following hold.

\medskip $(\mathrm{1})$ $X$ is a Mori dream space if and only if $\Cox(X, \LL)$ is finitely generated for some $\LL$.

\medskip $(\mathrm{2})$ If $X$ is a Mori dream space, then $X$ is a GIT quotient of $V = \Spec (\Cox (X))$ by the torus $G = \Hom (\NN^{\rho}, \GG_m)$ where $\rho$ is the Picard number of $X$.
\end{proposition}

The argument in \cite{Jo11} using GIT is based on Proposition \ref{hk2.9}.

\begin{example} A normal $\QQ$-factorial projective variety $X$ with Picard number $1$ is a Mori dream space if and only if $h^1(X, \OO_X) = 0$. Non-trivial examples of Mori dream spaces include projective toric manifolds (\cite[Corollary 2.4]{HK00}) and Fano manifolds (\cite[Corollary 1.3.2]{BCHM10}).

A K3 surface $X$ is a Mori dream space if and only if $\Eff(X)$ is rational polyhedral (\cite[Theorems 2.7]{AHL10}, \cite[Theorem 2.3]{Ot13}, \cite[Corollary 4.5]{McK10}). The latter is also equivalent to that the automorphism group $\Aut(X)$ of $X$ is finite. See \cite[Section 3]{Ca18} for more examples and non-examples.
\end{example}

Although Calabi--Yau varieties (see Definition \ref{caly} below) are not always Mori dream spaces, the cone conjecture due to Morrison and Kawamata together with the generalized abundance conjecture predict that a Calabi--Yau variety can be regarded as a ``Mori dream space'' modulo the action of (birational) automorphism groups. Before stating these two conjectures, we need a little more notions.

A \textit{minimal model} of $X$ is a normal projective $\QQ$-factorial variety $X^{\prime}$ with only terminal singularities and nef canonical divisor such that there is a birational map $\alpha: X^{\prime} \dashrightarrow X$. We call the pair $(X^{\prime}, \alpha)$ a \textit{marked minimal model} of $X$ with a \textit{marking} $\alpha$. Two minimal models of $X$ are said to be \textit{isomorphic}, if they are isomorphic as abstract varieties; while two marked minimal models $(X_i, \alpha_i)$ of $X$ are said to be \textit{isomorphic} if there exists an isomorphism $\beta : X_1 \to X_2$ such that $\alpha_1 = \alpha_2 \circ \beta$.

\begin{caution} For two marked minimal models $(X_i, \alpha_i)$ of $X$, even if $X_1$ and $X_2$ are isomorphic (as abstract varieties), $(X_1, \alpha_1)$ and $(X_2, \alpha_2)$ are not necessarily isomorphic (as marked minimal models).
\end{caution}

If $X$ itself is minimal, then $\alpha$ is an isomorphism in codimension $1$, and we obtain an isomorphism $\alpha_{\ast}: N^1(X^{\prime})_{\RR} \to N^1(X)_{\RR}$ preserving the movable cone and the effective movable cone. Let $\Aut(X)$ be the automorphism group and $\Bir(X)$ the birational automorphism group of $X$. Assume that $X$ is minimal, then there is a (contravariant) linear representation $\sigma: \Bir(X) \to \GL(N^1(X)_{\RR}, \RR)$ given by $g \mapsto g^{\ast}$.

\begin{definition}\label{caly} A normal projective variety $X$ with only $\QQ$-factorial terminal singularities is called \textit{Calabi--Yau}, if $h^1(X, \OO_X) = 0$ and $K_X$ is numerically trivial.
\end{definition}

Now we can state the cone conjecture due to Morrison \cite{Mo93, Mo96} and Kawamata \cite[Conjecture 1.12]{Ka97}.

\begin{conjecture}[Cone conjecture]\label{coneconj} Let $X$ be a Calabi--Yau variety. Then the following hold.

\medskip $\mathrm{(1)}$ The number of $\Aut(X)$-equivalence classes of faces of the cone $\Nefe(X)$ corresponding to birational contractions or fiber space structures is finite. Moreover, there exists a rational polyhedral cone $\Pi$ which is a fundamental domain for the action of $\Aut(X)$ on $\Nefe(X)$, in the sense that $\Nefe(X) =\bigcup_{g\in\Aut(X)}g^{\ast}\Pi$, and $\Pi^{\circ} \cap (g^{\ast}\Pi)^{\circ} = \varnothing$ unless $g^{\ast} = \mathrm{id}$.

\medskip $\mathrm{(2)}$ The number of the $\Bir(X)$-equivalence classes of chambers $\Nefe(X^{\prime})$ in the cone $\Move(X)$ for the marked minimal models $X^{\prime}$ is finite. In other words, the number of isomorphism classes of minimal models of $X$ is finite. Moreover, there exists a rational polyhedral cone which is a fundamental domain for the action of $\Bir(X)$ on $\Move(X)$.
\end{conjecture}

We refer to \cite{LOP18, To10} for a survey of this widely open conjecture. %The first statement of item (1) follows from the second statement on fundamental domains (see e.g., \cite{To10}). Indeed, each contraction of $X$ to a projective variety is given by some semi-ample divisor on $X$. The class of such a divisor in $N^1(X)_{\RR}$ lies in the effective nef cone. Moreover, two semi-ample divisors in the interior of the same face of some cone determine the same contraction of $X$, because they have degree zero on the same curves. On the other hand, the first statement of item (2) also follows from the second statement on fundamental domains. The proof is based on Shokurov's theory of geography of log models (\cite{SC11, KKL16}).
%\begin{theorem}[{\cite[Theorem 2.14]{CL14}}] Let $X$ be a Calabi--Yau variety. Assume that there exists a rational polyhedral cone which is a fundamental domain for the action of $\Bir(X)$ on $\Move(X)$, then the number of isomorphism classes of minimal models of $X$ is finite. \end{theorem}

The usual \textit{abundance conjecture} predicts that every effective nef divisor on a Calabi--Yau variety is semi-ample. The following is the so-called \textit{generalized abundance conjecture} for Calabi--Yau varieties (\cite[Question 2.6]{Og93}), which is open in dimension three as well; see  \cite{Og93, LOP18, LP20, LS20} for more discussions and recent progress. We note that every nef divisor on a Fano variety is semi-ample because of the basepoint-free theorem (see e.g., \cite[Theorem 6-2-1]{Ma02}).

\begin{conjecture}\label{abundconj} Every nef divisor on a Calabi--Yau variety is semi-ample.
\end{conjecture}

We remark the following Lefschetz type result for the nef cone  (\cite{Ko91}, \cite[Theorem 4.3]{HLW02}, \cite[Proposition 3.5]{BI09}).

\begin{theorem}\label{kover} Let $i: D \hookrightarrow Z$ be a smooth ample divisor in a smooth projective variety $Z$ with $\dim Z \geq 4$. Then $i_{\ast}(\NE(D))_{K_D\leq 0} = \NE(Z)_{K_D\leq 0}$. If in addition, $-(K_Z + D)$ is nef, then $Z$ is a Fano manifold and $i_{\ast}(\NE(D)) = \NE(Z)$, or dually, $i^{\ast}(\Nef(Z)) = \Nef(D)$.
\end{theorem}

\begin{proposition}[see e.g., {\cite[Theorem 3.1]{CO15}}]\label{co3.1} Let $Z$ be a Fano manifold of dimension $\geq 4$ and $X$ a smooth member of the linear system $|-K_{Z}|$. Let $i: X \hookrightarrow Z$ be the natural inclusion. Then $X$ is a smooth Calabi--Yau manifold in the strict sense, namely, $\omega_X \cong \OO_X$, $\pi_1(X) = \{1\}$, and $h^q(\OO_X) = 0$ for $0 < q < \dim X$. The nef cone $\Nef(X) = i^{\ast}(\Nef(Z))$ is rational polyhedral. Moreover, the automorphism group $\Aut(X)$ is finite.
\end{proposition}

%\medskip

\section{Proof of Theorem \ref{main}}\label{prof}

Let us start with a Lefschetz type theorem for nef cones whose proof is inspired by the argument in \cite{Ko91}.

\begin{proposition}\label{prop} Let $Z$ be a smooth Mori dream space of dimension $\geq 4$. Assume that each contraction of $Z$ admits a fiber of dimension $\geq 2$. Let $i: X \hookrightarrow Z$ be a smooth ample divisor in $Z$. Then the restriction map $i^{\ast}: \Nef(Z) \xrightarrow{\sim} \Nef(X)$ is an isomorphism.
\end{proposition}

\begin{proof} %Since $X$ is an ample divisor in $Z$ and $\dim X \geq 3$, by the Lefschetz hyperplane theorem, we have $h^1(X, \OO_X) = 0$, and
Since $X$ is an ample divisor in $Z$, $h^1(X, \OO_X) = 0$, and since $\dim X \geq 3$, by the Lefschetz hyperplane theorem, we have $i^{\ast}: \Pic(Z) \xrightarrow{\sim} \Pic(X)$, which induces $i^{\ast}: N^1(Z)_{\RR} \xrightarrow{\sim} N^1(X)_{\RR}$, or dually, $i_{\ast}: N_1(X)_{\RR} \xrightarrow{\sim} N_1(Z)_{\RR}$.

For our purpose, it is more convenient to use the closed cone of effective $1$-cycles, which is the dual of the nef cone. The linear map $i_{\ast}$ induces an inclusion $i_{\ast}(\NE(X))\subset \NE(Z)$. Let us note that $\NE(Z)$ is rational polyhedral since $Z$ is a Mori dream space. Thus, to show that $i_{\ast}$ is surjective, it suffices to show that all extremal rays of $\NE(Z)$ are contained in the image of $i_{\ast}$.

Take an extremal ray $R$ of $\NE(Z)$ and let $f: Z \to V$ be the contraction associated with $R$. This contraction exists because $\Nef(Z)$ is generated by semi-ample divisors. Our aim is to  show that $R$ is in the image of $i_{\ast}$, which is enough to show that $f$ contracts some curve on $X$. By our assumption, there exists a fiber $F$ of the contraction $f$ satisfying that $\dim F \geq 2$. Since $X|_F$ is of dimension $\geq 1$, we can find a curve $C$ in $X|_F$. This curve $C$ is contracted by $f$, so the class $[C]$ is in the extremal ray $R$. Since $C \subset X$,  we conclude that $R$ is in the image of $i_{\ast}: \NE(X) \subset \NE(Z)$.

This show that $i_{\ast}: \NE(X)\xrightarrow{\sim} \NE(Z)$ is an isomorphism. By taking dual, we obtain the desired isomorphism $i^{\ast}: \Nef(Z) \xrightarrow{\sim} \Nef(X)$.
\end{proof}

\begin{remark}
The condition that each contraction of $Z$ admits a fiber of dimension $\geq 2$ is necessary. In fact, there are ample hypersurfaces in the blow-up of $\PP^4$ at two distinct points or in $\PP^1 \times \PP^3$ with strictly larger nef cone than the ambient varieties. We refer to \cite{HLW02, Sz03} for the former and \cite{HLW02, Ot15} for the latter (see also Example \ref{ottem} below).
\end{remark}

Let us turn to prove our first main result Theorem \ref{main}. Recall the assumption that $Z$ is a smooth Mori dream space of dimension $\geq 4$ whose extremal contractions are of fiber type of relative dimension at least $2$ and $i: X \hookrightarrow Z$ is a smooth ample divisor in $Z$.

\begin{proof}[Proof of Theorem \ref{main}] By our assumption and Proposition \ref{prop}, we see that $i^{\ast}: \Nef(Z) \xrightarrow{\sim} \Nef(X)$ is an isomorphism. This implies that $\Nef(X)$ is a rational polyhedral cone generated by semi-ample divisors, since $Z$ is a Mori dream space. Let $f: Z \to V$ be the  contraction of $Z$ associated with an arbitrary extremal ray. Then the Stein factorization of the restriction morphism $f|_X$ is exactly the  contraction of $X$ associated with this extremal ray under the identification $i_{\ast}(\NE(Z)) = \NE(X)$. Since each extremal contraction of $Z$ is of fiber type of relative dimension at least $2$, we conclude that $\Nef(Z) = \Movo(Z) = \Eff(Z)$, and that each extremal contraction of $X$ is of fiber type. The latter implies that $\Nef(X) = \Movo(X) = \Eff(X)$. Therefore, $X$ is a Mori dream space, and moreover, $\Eff(Z) = \Movo(Z) = \Nef(Z) = i^{\ast}(\Nef(X)) = i^{\ast}(\Movo(X)) = i^{\ast}(\Eff(X))$.
\end{proof}

The following lemma is well-known; see for instance, \cite[Proposition 8]{Jo11}. We give an alternative proof here.

\begin{lemma}\label{lem2} Let $Z$ be a smooth projective variety which is a product of some Mori dream spaces with Picard number $1$. Then $Z$ is a Mori dream space satisfying that $\Nef(Z) \cong \Movo(Z) \cong \Eff(Z)$.
\end{lemma}

\begin{proof} Recall that for a smooth projective variety $W$ with Picard number $1$, $W$ is a Mori dream space if and only if $h^1(W, \OO_W) = 0$. We write $Z = W_1 \times \cdots \times W_k$, where $W_i$ are Mori dream spaces with Picard number $1$. Then $h^1(Z, \OO_Z) = 0$ by the K\"{u}nneth formula, and moreover,
\[ \Pic(Z) \cong \Pic(W_1) \times \cdots\times \Pic(W_k). \]
This implies that
\[ N^1(Z)_{\RR} \cong N^1(W_1)_{\RR} \times \cdots\times N^1(W_k)_{\RR}. \]
Let $H_i$ be the divisor in $Z$ which is the pullback of the ample generator of $\Pic(W_i)$. Then $\Nef(Z)$ is a rational polyhedral cone generated by these semi-ample divisors $H_i$. Since every extremal contraction of $Z$ is of fiber type, we see that $\Nef(Z) \cong \Movo(Z) \cong \Eff(Z)$, and $Z$ is a Mori dream space.
\end{proof}

\begin{corollary}\label{cor} Let $Z$ be a smooth projective variety which is a product of some Mori dream spaces, each having dimension $\geq n+1$ and Picard number $1$, where $n \geq 1$ is a positive integer. Let $X$ be a general complete intersection of $m \leq n$ smooth ample divisors in $Z$. Assume that $\dim X \geq 3$. Then $X$ is a Mori dream space and moreover,
\[ \Eff(Z) = \Movo(Z) = \Nef(Z) = i^{\ast}(\Nef(X)) = i^{\ast}(\Movo(X)) = i^{\ast}(\Eff(X)). \]
\end{corollary}

\begin{proof}[Proof of Corollary \ref{cor}] Let $X$ be a general complete intersection of $m \leq n$ smooth ample divisors $D_1, \dots, D_m$ in $Z$. Let $Y_i$ be the smooth complete intersection of $D_1, \dots, D_i$, where $1 \leq i \leq m$. Then by Theorem \ref{main}, $Y_1 = D_1$ is a Mori dream space whose extremal contractions are of fiber type of relative dimension at least $n$. By induction on $m$, we see that $Y_{m-1}$ is a Mori dream space of dimension $\geq 4$ whose extremal contractions are of fiber type of relative dimension at least $n-m+2 \geq 2$. By Theorem \ref{main} again, we conclude that $X$ is a Mori dream space, and moreover, $\Eff(Z) = \Movo(Z) = \Nef(Z) = i^{\ast}(\Nef(X)) = i^{\ast}(\Movo(X)) = i^{\ast}(\Eff(X))$.
\end{proof}

The pairs of $Z$ and $X$ appearing in Theorem \ref{main} and Corollary \ref{cor} give examples of the so-called \textit{birational twins} in the following sense of Lozano Huerta and Massarenti. We refer to \cite{LHM21} for an introduction and \cite{LHM20, LMR20} for more details.

\begin{definition}[{\cite[Definitions 1.1 and 1.2]{LHM20}}]\label{biratwin} Let $W$ and $Y$ be Mori dream spaces  and let $i: W \hookrightarrow Y$ be an embedding. We say that $W$ and $Y$ are \textit{birational twins} if the pull-back $i^{\ast}: \Pic(Y) \to \Pic(W)$ induces an isomorphism such that
\[ i^{\ast}(\Nef(Y)) = \Nef(W), \ i^{\ast}(\Movo(Y)) = \Movo(W), \ i^{\ast}(\Eff(Y)) = \Eff(W), \]
and in addition, $i^{\ast}$ preserves the Mori chamber decompositions of $\Eff(Y)$ and $\Eff(W)$.
\end{definition}

\begin{example}[\cite{Ot15}]\label{ottem} Let $n \geq 3$ be a positive integer. A very general hypersurface $D \subset \PP^1 \times \PP^n$ of degree $(d, e)$ with $d \geq n + 1$ and $e \geq 2$ is not a Mori dream space. Indeed, $\Nef(\PP^1 \times \PP^n)$ is generated by the rays $H_i = p_i^{\ast}\OO(1)$, where $p_1$ and $p_2$ are the two projections. By \cite[Theorem 1.1 (iv)]{Ot15}, we know that $\Nef(D)$ is generated by the rays $H_1$ and $neH_2 - dH_1$, which is strictly larger that $\Nef(\PP^1 \times \PP^n)$, and moreover, the divisor $neH_2 - dH_1$ has no effective multiple, in particular, it is not semi-ample.
\end{example}

\begin{question} Let $X$ be a smooth ample divisor in a Mori dream space $Z$ with $\dim Z \geq4$. Is the nef cone $\Nef(X)$ of $X$ always rational polyhedral?
\end{question}

\begin{remark} Let $D \subset Y$ be a smooth ample divisor in a smooth projective variety $Y$ with $\dim Y \geq 4$. Granting that $\Nef(Y) \cong \Nef(D)$, $\Movo(D)$ is strictly larger than $\Movo(Y)$ in general. It can be even worse that $\Movo(D)$ is irrational while $Y$ is a Mori dream space, for instance, when $D$ is a certain Calabi--Yau complete intersection in a product of projective spaces \cite{CO15, HT18, Og14, Og18, Ot15, Ya21} (see also Section \ref{prof2}).
\end{remark}

%\medskip

\section{General results}\label{gener}

In this section, we prove some general results on the cone conjecture and expect that these results help to further studies. Let us start with a few notations. Let $V$ be a finite-dimensional $\RR$-vector space admitting a distinguished $\QQ$-structure, i.e., there is a $\QQ$-vector space $V(\QQ)$ such that $V = V(\QQ)\otimes_{\QQ}\RR$. Let $C \subset V$ be a convex cone, and denote by $C^{+}$ the smallest convex cone in $V$ containing all the $\QQ$-rational points of $\overline{C}$, in other words, the convex hull of $\overline{C} \cap V(\QQ)$. %Clearly $C^+ = \overline{C}^+$ so \[ \Ampp(X) = \Nefp(X)\subset \Movp (X) = \Movcp (X). \]
The following result due to Looijenga (\cite[Proposition 4.1 and Application 4.14]{Lo14}) is very useful to the study of the cone conjecture.

\begin{proposition}\label{looij} Let $N$ be a finitely generated free $\ZZ$-module, and $C$ be an open strictly convex cone in the $\RR$-vector space $N_{\RR}:= N\otimes \RR$. Let $\Gamma$ be a subgroup of $\GL(N)$ which preserves the cone $C$. Suppose that there is a rational polyhedral cone $\Pi \subset C^+$ such that $C \subset \Gamma\cdot\Pi$. Then $C^+ = \Gamma\cdot\Pi$, and there exists a rational polyhedral fundamental domain for the action of $\Gamma$ on $C^{+}$. %Moreover, the group $\Gamma$ is finitely presented.
\end{proposition}

As applications of Proposition \ref{looij}, one can immediately conclude the following two results.

\begin{lemma}\label{ratnef} Let $X$ be a Calabi--Yau variety with rational polyhedral nef cone $\Nef(X)$. Then there is a rational polyhedral fundamental domain for the action of $\Aut(X)$ on  $\Nefp(X) = \Nef(X)$.
\end{lemma}

\begin{proposition}\label{mvsk} Let $X$ be a Calabi--Yau variety. Assume there is a rational polyhedral fundamental domain for the action of $\Aut(X)$ on $\Nefe(X)$, then $\Nefp(X) = \Nefe(X)$.
\end{proposition}

\begin{proof} This follows from Proposition \ref{looij} and Proposition \ref{lop2.15} below.
\end{proof}

\begin{proposition}[{\cite[Theorem 2.15]{LOP20}}]\label{lop2.15} Let $X$ be a Calabi--Yau variety. Then the inclusion $\Nefe(X) \subset \Nefp(X)$ holds.
\end{proposition}

\begin{remark} The generalized abundance conjecture (Conjecture \ref{abundconj}) also predicts the equality $\Nefp(X) = \Nefe(X)$.
\end{remark}

In order to study the structure of movable cones, we note some basic facts (see e.g., \cite{Ka88, Ka97}).

\begin{proposition}\label{ka2.3} Let $X$ be a Calabi--Yau variety. Then
\[ \Movox \subset \bigcup_{(X^{\prime}, \alpha)} \alpha_{\ast} \Nefe(X^{\prime}) \subset \Move(X), \]
and
\[ \Movox \subset \bigcup_{(X^{\prime}, \alpha)} \alpha_{\ast} \Nefp(X^{\prime}) \subset \Movp(X), \]
where the union on the right hand side is taken for all the marked minimal models $(X^{\prime}, \alpha)$ of $X$ up to isomorphism.
\end{proposition}

\begin{proof} The second and the fourth inclusions are clear. In view of Proposition \ref{lop2.15}, it is enough to show the first inclusion $\Movox \subset \bigcup_{(X^{\prime}, \alpha)} \alpha_{\ast} \Nefe(X^{\prime})$. Take an arbitrary element in $\Movox$ which is represented by a big $\RR$-divisor on $X$, say $D$. We can choose a small positive real number $\varepsilon > 0$ such that the pair $(X, \varepsilon D)$ is Kawamata log terminal. %We note that by our definition of Calabi--Yau varieties, the canonical divisor $K_X$ is trivial (not only numerically trivial).
Since $D$ is big and $K_X$ is numerically trivial, by \cite[Theorem 1.2]{BCHM10}, we can run a minimal model program for the pair $(X, \varepsilon D)$ to obtain a new pair $(X^{\prime}, D^{\prime})$ such that $K_{X^{\prime}} + D^{\prime}$ is nef. Since $D$ is movable, each step of the minimal model program for $(X, \varepsilon D)$ is a flop (\cite[Lemma 2.3]{Ka88}). This implies that $X^{\prime}$ is a minimal model of $X$, in particular, $K_{X^{\prime}}$ is numerically trivial. Hence $D^{\prime}$ is nef.
\end{proof}

The next result is conditional and we have to assume \textit{Flop Conjecture II} holds, that is, any sequence of $D$-flops terminates (see \cite{Ma91}). Note that nowadays we know the existence of $D$-flops (or more generally, $D$-flips; see \cite{BCHM10} and references therein).

\begin{proposition}\label{ka2.32} Let $X$ be a Calabi--Yau variety. Assume that any sequence of $D$-flops on $X$ terminates. Then
\[ \Move(X) = \bigcup_{(X^{\prime}, \alpha)}\alpha_{\ast}\Nefe(X^{\prime}), \]
where the union on the right hand side is taken for all the marked minimal models of $X$ up to isomorphism.
\end{proposition}

\begin{proof} \cite[Theorem 2.3]{Ka97} proves the case of Calabi--Yau $3$-folds. The same proof works assuming that any sequence of $D$-flops on $X$ terminates.
\end{proof}

In dimension three and four, any sequence of $D$-flops terminates (see \cite{Ma91} and references therein), while in higher dimension, it is largely open in general. We mention the following rather special termination result, which is precisely sufficient for our purpose.

\begin{definition} Let $X$ be a normal variety and $D$ a Weil divisor on $X$, such that $K_X + D$ is $\RR$-Cartier. For a birational morphism $\mu: X^{\prime} \to X$ from a normal variety $X^{\prime}$ and a prime Weil divisor $E^{\prime}$ on $X^{\prime}$, we define the \textit{log discrepancy} $a(E^{\prime}; X, D)$ by
\[ a(E^{\prime}; X, D) := (\mathrm{The \ coefficient \ of} \ E^{\prime} \ \mathrm{in} \ K_{X^{\prime}} - \mu^{\ast}(K_X + D)) + 1. \]
For a proper closed subset $W$ of $X$, we define the \textit{minimal log discrepancy} of $(X,D)$ by
\[ \mld(W; X, D) := \inf_{\mu(E^{\prime})\subset W} a(E^{\prime}; X, D). \]
\end{definition}

\begin{proposition}\label{mz4.1} Let $X$ be a smooth Calabi--Yau variety satisfying that each of its minimal models is smooth. Let $D$ be an effective $\RR$-divisor on $X$ such that the pair $(X, D)$ has at worst log canonical singularities. Then any sequence of $D$-flops on $X$ terminates.
\end{proposition}

\begin{proof} This is implied in the proof of \cite[Theorem 4.1]{MZ13}, which we will recall here for the sake of completeness. Consider the following sequence of $D$-flops on $X$:
\[ \xymatrix{
  X:= X_1 \ar@{-->}[rr]^{\phi_1} \ar[dr]_{f_1}
  &  &    X_2 \ar[dl]^{f_1^{+}} \ar@{-->}[rr]^{\phi_2} \ar[dr]_{f_2} &  &    X_3 \ar[dl]^{f_2^{+}} \ar@{-->}[rr]^{}
  &  &   \cdots   \\
  & Z_1 & & Z_2              }.
\]
Here $D_1 := D$ and $D_{i+1}$ is the proper transform of $D_{i}$ by $\phi_{i}$ for every $i \geq 1$. By \cite{Sh04}, it is enough to show the following two statements.

\medskip $\mathrm{(1)}$ For each $i \geq 1$, the function $p_i$ on $X_i$ defined by
\[ p_i(x) = \mld(x; X_i, D_i), \]
where $x \in X_i$, is lower semi-continuous.

\medskip $\mathrm{(2)}$ Let $\mathfrak{S}$ be the set of the minimal log discrepancies defined by
\[ \mathfrak{S} := \bigcup_{i \geq 1} \mld(W_i; X_i, D_i), \]
where $W_i$ is the exceptional locus of $f_i$. The set $\mathfrak{S}$ satisfies the ascending chain condition.

\medskip For item $\mathrm{(1)}$, since every $X_i$ is smooth by our assumption, each $p_i$ is lower semi-continuous by \cite[Theorem 4.4]{EMY03}. For item $\mathrm{(2)}$, we first note that $\mld(W_i; X_i, D_i) \leq \dim X_i$ since all $X_i$ are smooth. On the other hand, all pairs $(X_i, D_i)$ still have only log canonical singularities since each $f_i$ is the contraction of a $(K_{X_i} + D_i)$-negative extremal ray. Therefore, $\mld(W_i; X_i, D_i) \geq 0$. Moreover, since all $X_i$ are smooth and the sets of all coefficients of $D_i$ are stable, the set of log discrepancies of $(X_i, D_i)$ is discrete by \cite[Theorem 5.2]{Ka14}. So $\mathfrak{S}$ is a finite set and the assertion follows.
\end{proof}

Now we are ready to show that, if a Calabi--Yau variety admits only finitely many minimal models up to isomorphism, and for each of these minimal models, the cone conjecture for the nef cone holds, then the cone conjecture for the movable cone also holds for the given Calabi--Yau variety.

\begin{proposition}\label{techcoh} Let $X$ be a Calabi--Yau variety. Assume that any sequence of $D$-flops on $X$ terminates. Suppose there are only finitely many minimal models of $X$ up to isomorphism, say, $X_0, \dots, X_k$. Suppose further that for each $X_i$, there is a rational polyhedral cone $\Pi_i\subset \Nefe(X_i)$ such that $\Aut(X_i)\ast\Pi_i = \Nefe(X_i)$. Here $\Pi_i$ is not necessarily a fundamental domain. Then $\Move(X) = \Movp(X)$ and the cone conjecture holds for $\Move(X)$.
\end{proposition}

\begin{proof} We fix a birational map $\alpha_i: X_i \dashrightarrow X$ for each $X_i$. Every $\alpha_{i{\ast}}\Pi_i$ is rational polyhedral, so the convex hull $\Pi$ of the finite union $\bigcup_i\alpha_{i{\ast}}\Pi_i$ is again rational polyhedral. By Proposition \ref{ka2.32},
\[ \Move(X) = \bigcup_{(X^{\prime}, \alpha)}\alpha_{\ast}\Nefe(X^{\prime}), \]
where the union on the right hand side is taken for all the marked minimal models $(X^{\prime}, \alpha)$ of $X$ up to isomorphism.

On the other hand, $\Aut(X_i)\ast\Pi_i = \Nefe(X_i)$ and
\[ \Pi_i \subset \Nefe(X_i) \subset \Nefp(X_i) \]
for each $i$. The last inclusion follows from Proposition \ref{lop2.15}. Therefore, $\Move(X) = \Bir(X)\ast \Pi$, and $\Pi \subset \Movp(X)$.

The two cones $\Move(X)$ and $\Movp(X)$ share the same interior $\Movox$. By Proposition \ref{looij}, we obtain that
\[ \Movp(X) = \Bir(X)\ast \Pi = \Move(X) \]
and that there is a rational polyhedral fundamental domain for the action of $\Bir(X)$ on $\Movp(X) = \Move(X)$.
\end{proof}

Similarly, we can prove the following unconditional result.

\begin{proposition}\label{techcog} Let $X$ be a Calabi--Yau variety. Suppose there are only finitely many minimal models of $X$ up to isomorphism, say, $X_0, \dots, X_k$. Suppose further that for each $X_i$, there is a rational polyhedral cone $\Pi_i\subset \Nefp(X_i)$ such that $\Aut(X_i)\ast\Pi_i = \Nefp(X_i)$. Then there is a rational polyhedral fundamental domain for the action of $\Bir(X)$ on $\Movp(X)$.
\end{proposition}

%\medskip

\section{Proof of Theorem \ref{main2}}\label{prof2}

\begin{proof}[Proof of Theorem \ref{main2}] Let $Z$ be a Fano manifold of dimension $\geq 4$ whose extremal contractions are of fiber type. Let $X \in |-K_Z|$ be a smooth member, which is a smooth Calabi--Yau variety (Proposition \ref{co3.1}). By Propositions \ref{mz4.1} and \ref{techcoh}, it suffices to show that $X$ has the unique minimal model itself up to isomorphism. Recall that by a result of Kawamata (\cite[Theorem 1]{Ka08}), any birational map between Calabi--Yau varieties is decomposed into finitely many flops up to automorphisms of the target variety. Thus, it is enough to investigate the small contraction and the associated flops of $X$. By \cite[Theorem 5.7]{Ka88}, any small contraction of a Calabi--Yau variety is given by a codimension one face of $\Nef(X)$  up to automorphisms of $X$.

By Theorem \ref{kover}, $\NE(Z) \cong \NE(X)$, or dually, $\Nef(X) \cong \Nef(Z)$. Let $f: Z \to V$ be the contraction of $Z$ associated with an arbitrary extremal ray. Then the Stein factorization $g: X \to W$ of the restriction morphism $f|_X$ is exactly the contraction of $X$ associated with this extremal ray under the identification $\NE(Z) \cong \NE(X)$. As discussed above, we may assume that the elementary extremal contraction $g$ is small. By assumption, $Z$ has finitely many elementary extremal contractions which are all of fiber type.

If $\dim V \leq \dim Z - 2$, then $g$ must be of fiber type. Therefore, we may assume that $\dim V = \dim Z - 1$. Then there is an open dense subset $U \subset V$ such that $f^{-1}(U) \to U$ is equidimensional (each fiber has dimension one). By \cite[Theorem 3.1]{An85} (and its proof), $f^{-1}(U) \to U$ is a conic bundle, so a general fiber intersects $X$ at two points. It follows that $f|_X$ is a ramified double cover. Let $\theta: X \dashrightarrow X$ be the associated covering involution, which is not an isomorphism because $f|_X$ is ramified. Since $g$ is small by our assumption, $\theta$ is isomorphic in codimension one. From $\rho(X/W) = 1$, it follows that $\theta: X \dashrightarrow X$ is the flop of the small extremal contraction $g: X \to W$. This completes the proof.
\end{proof}

\begin{example}\label{homog} Let $Z$ be a rational homogeneous manifold of dimension $\geq 4$. One can write $Z = G/P$ where $G$ is a semi-simple Lie group and $P$ is a parabolic subgroup of $G$. It is known that $Z$ is a Fano manifold and all extremal contractions of $Z$ are of fiber type (see e.g., \cite[Proposition 2.3 and Proposition 2.6]{M15}), so one can apply Theorem \ref{main2}. %Let us note that hypersurfaces in certain rational homogeneous manifolds give pairs of Calabi--Yau manifolds which are derived equivalent but not birationally equivalent (see \cite{Ma20} and references therein). By our result, such pairs of Calabi--Yau manifolds are birationally equivalent if and only if they are isomorphic as abstract varieties.
\end{example}

\begin{example}\label{prod} Let $Z:= Z_1 \times \cdots \times Z_n$ be the product of Fano manifolds $Z_i$ with Picard number $1$. Then clearly $Z$ is a Fano manifold and all extremal contractions of $Z$ are of fiber type, so one can apply Theorem \ref{main2} whenever $\dim Z \geq 4$. More generally, for Fano manifolds $Z_i$ whose extremal contractions are of fiber type, the product $Z_1 \times \cdots \times Z_n$ is again a Fano manifold with all extremal contractions of fiber type (see e.g., \cite[Lemmas 4.1 and 4.4]{Dr16}). We also refer to \cite{Dr16, Ou18} for more such examples.
\end{example}

\begin{example}\label{compl} Let $\PP := \PP^{n_1} \times \cdots \times \PP^{n_k}$ be a product of projective spaces such that $\dim \PP = \sum n_i \geq 4$. Let $X$ be a general Calabi--Yau complete intersection of $m$ ample divisors in $\PP$ with $m \leq \min\{n_i\}$. One can view $X$ as a smooth member of $|-K_Z|$, where $Z$ is a general Fano complete intersection of $m-1$ ample divisors in $\PP$. Clearly,
\[ \dim Z = \dim \PP - (m-1) \geq \dim \PP + 1 - \min\{n_i\}. \]
In particular, all extremal contractions of $Z$ are of fiber type. Thus, one can apply Theorem \ref{main2} to this $X$, which covers the results from \cite{Ya21} and \cite{CO15, Og14, Og18, Sk17}. We note that the arguments from \cite{CO15, Ya21} are of independent interest which is a combination of the minimal model program and the classical theory of Coxeter groups.
\end{example}

%\begin{example} Let $n$ and $N$ be two positive integers. Let $Z$ be a general complete intersection of $n-1$ hypersurfaces of multidegree $(1, \dots, 1)$ in $(\PP^n)^N = \PP^n \times \cdots \times \PP^n$. Then $Z$ is a Fano manifold and $\Nef((\PP^n)^N) \cong \Nef(Z)$. Since $\dim Z = n(N-1)+1 = \dim (\PP^n)^{N-1} + 1$, every extremal contraction of $Z$ are of fiber type. A general member $X$ of $|-K_Z|$ is a smooth hypersurface in $Z$ of multidegree $(2, \dots, 2)$, thus, one can apply Theorem \ref{main2} whenever $\dim X \geq 3$. Note that $X$ can be viewed as a general complete intersection of $n-1$ hypersurfaces of multidegree $(1, \dots, 1)$ and one hypersurface of multidegree $(2, \dots, 2)$ in $(\PP^n)^N$, so we cover the results from \cite{CO15, Og14, Og18, Sk17}.                                      \end{example}

\begin{remark}\label{htsec3} Let $X$ be a general complete intersection of $5$ hypersurfaces of bidegree $(1, 1)$ in $\PP^4 \times \PP^4$. Then $X$ can be viewed as a general anti-canonical hypersurface in $Z$, where $Z$ a general complete intersection of $4$ hypersurfaces of bidegree $(1, 1)$ in $\PP^4 \times \PP^4$. Note that in this case, $Z$ admits divisorial contractions so that Theorem \ref{main2} cannot be applied. In fact, $X$ may admit non-trivial minimal models other than itself. Nevertheless, it is shown that $X$ admits only finitely many minimal models up to isomorphism which are all smooth (\cite[Section 3]{HT18}; see also \cite{LW21}), so the general results Propositions \ref{mz4.1} and \ref{techcoh} can still be applied to deduce the movable cone conjecture for $X$.
\end{remark}

\begin{remark} Consider $Z := Z_1 \times \cdots \times Z_N$ a product of Fano manifolds $Z_i$ with Picard number one. Denote by $n_i$ the dimension of $Z_i$ and by $r_i$ the index of $Z_i$, that is, the largest integer such that $-K_{Z_i} = r_i H_i$ for some ample divisor $H_i\in \Pic(Z_i)$. It is well-known that $r_i \leq n_i + 1$. Moreover, $r_i = n_i + 1$ if and only if $Z_i \cong \PP^{n_i}$ (\cite[Page 32]{KO73}); and $r_i = n_i$ if and only if $Z_i \cong \mathcal{Q}^{n_i}$ (\cite[Page 37]{KO73}), where $\mathcal{Q}^{n_i} \subset \PP^{n_i+1}$ is the smooth hyperquadric of dimension $n_i$.

In what follows, we assume that the linear system $|H_i|$ is base-point free for every $i$ for simplicity. Without loss of generality, we may assume that
\[ n_1 \leq n_2 \leq \cdots \leq n_N. \]
Consider a general complete intersection $X$ of $m$ ample divisors in $Z$ with trivial canonical divisor. Then $X$ is a smooth Calabi--Yau manifold in the strict sense. We note that for every $i = 1, \dots, N$,
\[ m \leq r_i \leq n_i + 1, \]
and hence
\[ \dim X = \dim Z - m \geq \dim Z - n_i - 1. \]
On the other hand,
\[ \dim Z_1 \times \cdots \widehat{Z_i} \times \cdots \times Z_N = \dim Z - n_i \leq \dim Z - n_1. \]
If $\dim X > \dim Z - n_i$, then the contraction of $X$ induced by
\[ Z \to Z_1 \times \cdots \widehat{Z_i} \times \cdots \times Z_N \]
is of fiber type. In particular, if $\dim X > \dim Z - n_1$, then all contractions of $X$ are of fiber type and $X$ admits the unique minimal model itself up to isomorphism.

Consider the case where $\dim X \leq \dim Z - n_i$. This implies that
\[ n_i \leq m \leq n_i + 1, \]
and thus, either $m = n_i$ or $m = n_i + 1$. It is not difficult to see that $X$ is a complete intersection of $m$ ample hypersurfaces in the product space                                                      \[ Z = (\PP^{m-1})^k \times (\PP^{m})^j \times (\mathcal{Q}^m)^{i-j-k} \times Z_{i+1}\times \cdots \times Z_N \]                                             with $\dim Z_{i+1} > m$. Notice that $X$ can be regarded as a smooth member of $|-K_Y|$, where $Y$ is a smooth Fano manifold which is a complete intersection of $m-1$ ample hypersurfaces in $Z$. If $k = 0$, i.e., $Z = (\PP^{m})^j \times (\mathcal{Q}^m)^{i-j} \times Z_{i+1}\times \cdots \times Z_N$, then clearly every contraction of $Y$ is of fiber type, and therefore, $X$ has the unique minimal model itself up to isomorphism by Theorem \ref{main2}. When $k \geq 1$, $Y$ may admit divisorial contractions which makes the problem more complicated. In Theorem \ref{cipp}, we handle a special case where $Z$ is a self-product of a projective space. The problem remains open in general. 
\end{remark}

%\medskip

\section{The construction for a special case}\label{spec}

In this section, we construct the minimal models for a special case, that is, the case where $N = 3$. This might help readers to understand the general construction presented in the next section.

\begin{example} Consider the general complete intersection $X$ given by $n + 1$ hypersurfaces $D_0, D_1, \dots, D_n$ of tridegree $(1, 1, 1)$ in $\PP^n_1\times \PP^n_2 \times \PP^n_3$, which is a smooth Calabi--Yau manifold in the strict sense of dimension $2n - 1$ (Proposition \ref{co3.1}). We assume that $\dim X \geq 3$.

Let $\mathbf{x}^i:= [x^i_0 : x^i_1 : \cdots : x^i_n]$ be the homogeneous coordinates of $\PP^n_i$, $i = 1, 2, 3$. We can write these hypersurfaces $D_j$ of tridegree $(1, 1, 1)$ by the following equations
\[ \sum_{0 \leq m_1, m_2, m_3 \leq n}a^{m_1 m_2 m_3}_jx^1_{m_1} x^2_{m_2} x^3_{m_3} = 0. \]

We add one more projective space $\PP^n_0 = \PP^n$ with homogeneous coordinates $\mathbf{x}^0:= [x^0_0 : x^0_1 : \cdots : x^0_n]$ and  define
\[ b^{m_0 m_1 m_2 m_3}:= a_{m_0}^{m_1 m_2 m_3} \]
for each $0\leq m_0, m_1, m_2, m_3\leq n$. Consider the following homogeneous form
\[ \sum_{0\leq m_0, m_1, m_2, m_3\leq n}b^{m_0 m_1 m_2 m_3}x^0_{m_0} x^1_{m_1} x^2_{m_2} x^3_{m_3} \]
in the product space $\PP^n_0 \times \PP^n_1 \times \PP^n_2 \times \PP^n_3$. We define the complete intersection $X_3$ of the following $n + 1$ hypersurfaces $D^3_{m_3}\,(m_3 = 0, 1, \dots, n)$ given by
\[ \sum_{0 \leq m_0, m_1, m_2 \leq n}b^{m_0 m_1 m_2 m_3}x^0_{m_0} x^1_{m_1} x^2_{m_2} = 0 \]
in the space $\PP^n_0 \times \PP^n_1 \times \PP^n_2$. Similarly, we define the complete intersection $X_2$ of the following $n + 1$ hypersurfaces $D^2_{m_2}\,(m_2 = 0, 1, \dots, n)$ given by
\[ \sum_{0 \leq m_0, m_1, m_3 \leq n}b^{m_0 m_1 m_2 m_3}x^0_{m_0} x^1_{m_1} x^3_{m_3} = 0 \]
in $\PP^n_0 \times \PP^n_1 \times \PP^n_3$, the complete intersection $X_1$ of the following $n + 1$ hypersurfaces $D^1_{m_1}\,(m_1 = 0, 1, \dots, n)$ given by
\[ \sum_{0 \leq m_0, m_2, m_3 \leq n}b^{m_0 m_1 m_2 m_3}x^0_{m_0} x^2_{m_2} x^3_{m_3} = 0 \]
in $\PP^n_0 \times \PP^n_2 \times \PP^n_3$, and the complete intersection $X_0$ of the following $n + 1$ hypersurfaces $D^0_{m_0}\,(m_0 = 0, 1, \dots, n)$ given by
\[ \sum_{0 \leq m_1, m_2, m_3 \leq n}b^{m_0 m_1 m_2 m_3}x^1_{m_1} x^2_{m_2} x^3_{m_3} = 0 \]
in $\PP^n_1 \times \PP^n_2 \times \PP^n_3$. We note that $D_j$'s are exactly $D^0_i$'s and $X = X_0$. Moreover, we can choose $X$ general \textit{a priori} such that $X_{\ell}$ is smooth for each $\ell = 0$, $1$, $2$, $3$.

We can also write the defining equations of $D_j$ for $j = 0, \dots, n$ as
\[ \sum_{0\leq k \leq n}\left(\sum_{0\leq m_2, m_3 \leq n} a^{k m_2 m_3}_j x^2_{m_2} x^3_{m_3} \right) x^1_k = 0. \]
Set the matrix
\[ A_1:= \left(\sum_{0\leq m_2, m_3 \leq n} a^{k m_2 m_3}_j x^2_{m_2} x^3_{m_3} \right)_{j, k}, \]
then
\[ A_1\cdot {}^t(x^1_0, x^1_1, \cdots, x^1_n) = 0 \]
gives equations of $n + 1$ hypersurfaces $D_j$, where ${}^t$ denotes the transpose as usual. We define the determinantal hypersurface
\[ W_1:= \left(\det(A_1) = 0 \right) \subset \PP^n_2 \times \PP^n_3. \]
For the complete intersection $X_1$, we can write $n + 1$ hypersurfaces $D^1_{m_1}$ as
\[ \sum_{0\leq m_0 \leq n} \left(\sum_{0 \leq m_2, m_3 \leq n} b^{m_0 m_1 m_2 m_3} x^2_{m_2} x^3_{m_3} \right) x^0_{m_0} = 0. \]
If we set the matrix
\[ A^{\prime}_1:= \left(\sum_{0 \leq m_2, m_3 \leq n} b^{m_0 m_1 m_2 m_3} x^2_{m_2} x^3_{m_3} \right)_{m_1, m_0}, \]
then
\[ A^{\prime}_1 \cdot {}^t(x^0_0, x^0_1, \dots, x^0_n) = 0 \]
gives equations of $n + 1$ hypersurfaces $D^1_{m_1}$. We define the determinantal hypersurface
\[ W^{\prime}_1:= (\det(A^{\prime}_1) = 0) \subset \PP^n_2 \times \PP^n_3. \]
Since by definition
\[ b^{m_0 m_1 m_2 m_3} = a_{m_0}^{m_1 m_2 m_3}, \]
we know that indeed
\[ A^{\prime}_1 = {}^tA_1, \]
and hence
\[ W^{\prime}_1 = W_1. \]
Moreover, we have natural projections $\alpha_1: X \to W_1$ and $\beta_1: X_1 \to W_1$, which further induce the birational map $X \dashrightarrow X_1$ and the following commutative diagram
\[ \xymatrix{
  X \ar@{-->}[rr]^{} \ar[dr]_{\alpha_1}
                &  &    X_1 \ar[dl]^{\beta_1}    \\
                & W_1                 }.
\]
Similarly, we can define the determinantal hypersurfaces $W_2 \subset \PP^n_1 \times \PP^n_3$ and $W_3 \subset \PP^n_1 \times \PP^n_2$, and moreover, we have the following commutative diagrams
\[ \xymatrix{
  X \ar@{-->}[rr]^{} \ar[dr]_{\alpha_2}
                &  &    X_2 \ar[dl]^{\beta_2}    \\
                & W_2                 },
\]
\[ \xymatrix{
  X \ar@{-->}[rr]^{} \ar[dr]_{\alpha_3}
                &  &    X_3 \ar[dl]^{\beta_3}    \\
                & W_3                 }.
\]
In fact, one can obtain analogous birational maps $X_j \dashrightarrow X_i$ for every pair of $i \neq j \in \{0, 1, 2, 3 \}$, where $X_0 = X$.

We will show in Section \ref{prof3} that such diagrams exhaust all the flops of $X$ (under certain assumptions). Therefore, we obtain all the minimal models of $X$ up to isomorphism by a result of Kawamata (\cite{Ka08}) and the symmetry of our construction.
\end{example}

%\medskip

\section{Proof of Theorem \ref{cipp}}\label{prof3}

In this section, we prove the last main result of this paper (Theorem \ref{cipp}).
%\begin{theorem}[$=$ Theorem \ref{cipp}]\label{corol} Let $X$ be a general Calabi--Yau complete intersection of $n+1$ hypersurfaces of multidegree $(1,\dots,1)$ in $P^n(N):= \PP^n_1 \times \cdots \times \PP^n_N$ of dimension $\geq 3$. Then $X$ has finitely many minimal model up to isomorphism and there exists a rational polyhedral cone which is a fundamental domain for the action of $\Bir(X)$ on $\Move(X)$.                                             \end{theorem}
The argument here is inspired by \cite[Section 3]{HT14} (see also \cite[Section 2]{Fr01}, \cite[Section 3.2]{Li16}, \cite[Section 3]{HT18}, \cite{LW21}). The method traces back to Cayley's work on determinantal quartic surfaces \cite{Ca69}. See \cite{FGGL13} for more details on the history and \cite{Og15, Og17} for related works.

\begin{setup}\label{mysetup} We write $P := P^n(N) := \PP^n_1 \times \cdots \times \PP^n_N$ for convenience, and always assume that $N \geq 2$. Let
\[ \pi_i: P \to \PP^n_1 \times \cdots \times \widehat{\PP^n_i} \times \cdots \times \PP^n_N =: Q_i \cong P^n(N-1) \]
be the natural projection. Let us note that $\dim P = nN$, while $\dim Q_i = n(N-1)$ for $i = 1, \dots, N$. We denote by $H_i$ the pullback of a hyperplane in $\PP^n_i$. Then clearly
\[ \Pic(P) = \langle H_1, \dots, H_N\rangle, \]
and
\[ \Eff(P) = \Movo(P) = \Nef(P) = \Cone(H_1, \dots, H_N), \]
where ``Cone'' denotes the convex cone generated by these elements.

Assume that the general complete intersection $X$ is given by $n+1$ hypersurfaces $D_0, D_1, \dots, D_n$ of multidegree $(1,\dots,1)$ in $P$. It follows that $X$ is a smooth Calabi--Yau manifold in the strict sense of dimension $n(N-1)-1$, and the nef cone
\[ \Nef(X)\cong \Nef(P) = \Cone(H_1, \dots, H_N) \]
is rational polyhedral from Proposition \ref{co3.1}. Moreover,
\[ \Nefe(X) = \Nefp(X) = \Nef(X). \]
\end{setup}

\begin{construction}\label{constr} Let $\mathbf{x}^i:= [x^i_0 : x^i_1 : \cdots : x^i_n]$ be the homogeneous coordinates of $\PP^n_i$, $i = 1, \dots, N$. We use the following notation for convenience:
\[ \II:= \{(m_1, m_2, \dots, m_N)\,|\,0\leq m_k \leq n \ \mathrm{for \ each} \ k\}. \]
%For an index $I \in \II$, set $\bfx_I:= x^1_{m_1} x^2_{m_2} \cdots x^N_{m_N}$, and $a^I_j:= a^{m_1\cdots m_N}_j \in \CC$, where $j\in \{0, 1, \dots, n\}$ is anthor independent index marking $n+1$ hypersurfaces. \sum_{I\in \II}a^I_j\bfx_I
We can write these hypersurfaces $D_j$ of multidegree $(1, \dots,1)$ by the following equations
\[ \sum_{(m_1, m_2, \dots, m_N)\in \II}a^{m_1\cdots m_N}_jx^1_{m_1} x^2_{m_2} \cdots x^N_{m_N} = 0. \]

%\begin{example} When $N = 2$, the $n+1$ hypersurfaces $D_j$ of multidegree $(1, \dots,1)$ are given by the equations     \[ \sum_{0\leq m_1, m_2 \leq n} a_j^{m_1 m_2} x^1_{m_1}x^2_{m_2} = 0 \]

%When $n = 1$, the two hypersurfaces $D_j$ of multidegree $(1, \dots,1)$ are given by the equations                         \[ \sum_{0\leq m_k\leq 1}a_j^{m_1\cdots m_N} x^1_{m_1} x^2_{m_2} \cdots x^N_{m_N} = 0 \] \end{example}

We add one more projective space $\PP^n_0 = \PP^n$ with homogeneous coordinates $\mathbf{x}^0:= [x^0_0 : x^0_1 : \cdots : x^0_n]$ and  define
\[ b^{m_0 m_1 m_2 \dots m_N}:= a_{m_0}^{m_1 m_2 \dots m_N} \]
for each $(m_1, m_2, \dots, m_N)\in \II$ and $0\leq m_0\leq n$. Consider the following homogeneous form
\[ \sum_{0\leq m_k \leq n,\,0\leq k \leq N}b^{m_0 m_1\cdots m_N}x^0_{m_0} x^1_{m_1} x^2_{m_2} \cdots x^N_{m_N}
\]
in the product space $\PP^n_0 \times \PP^n_1 \cdots \times \PP^n_N = \PP^n_0 \times P \cong P^n(N+1)$.

For each $\ell \in \{0, 1, \dots, N\}$, we define the complete intersection $X_{\ell}$ of the following $n+1$ hypersurfaces $D^{\ell}_{m_{\ell}}\,(m_{\ell}=0, \dots, n)$ given by
\[ \sum_{(m_0, m_1, \dots, m_{\ell-1}, m_{\ell+1}, \dots, m_N)\in \II}b^{m_0 m_1\cdots m_N}x^0_{m_0} x^1_{m_1} \cdots x^{\ell-1}_{m_{\ell-1}} x^{\ell+1}_{m_{\ell+1}} \cdots x^N_{m_N} = 0
\]
in the space $\PP^n_0 \times \cdots \times \widehat{\PP^n_{\ell}} \times \cdots \times \PP^n_N \cong P$.
\end{construction}

Let us note that $D_j$'s are exactly $D^0_j$'s and $X = X_0$. Moreover, we can choose $X$ general\footnote{``General'' means lying in some Zariski-dense open subset of the parameter space.} \textit{a priori} such that

\begin{assumption}\label{assump} $X_{\ell}$ is smooth for each $\ell = 0, 1, \dots, N$.
\end{assumption}

%In what follows, we will work under this assumption.
By Proposition \ref{co3.1}, these $X_{\ell}$ are also smooth Calabi--Yau manifolds in the strict sense satisfying $\Nef(X_{\ell}) \cong \Nef(P)$, and moreover,
\[ \Nefe(X_{\ell}) = \Nefp(X_{\ell}) = \Nef(X_{\ell}). \]

For each $i \in \{1, 2, \dots, N\}$, we can also write the defining equations of $D_j$ for $j = 0, \dots, n$ as the zero locus of
\[ \sum_{0\leq k \leq n}\left(\sum_{\substack{(m_1, \dots, m_{i-1}, \\ m_{i+1}, \dots, m_N)}} a^{m_1\cdots m_{i-1} k m_{i+1} \cdots m_N}_jx^1_{m_1} \cdots x^{i-1}_{m_{i-1}} x^{i+1}_{m_{i+1}} \cdots x^N_{m_N}\right) x^i_k. \]
Set the matrix
\[ A_i:= A_i(\mathbf{x}^1, \dots, \mathbf{x}^{i-1}, \mathbf{x}^{i+1}, \dots, \mathbf{x}^N):= \]
\[ \left(\sum a^{m_1\cdots m_{i-1} k m_{i+1} \cdots m_N}_jx^1_{m_1} \cdots x^{i-1}_{m_{i-1}} x^{i+1}_{m_{i+1}} \cdots x^N_{m_N}\right)_{j, k}, \]
then
\[ A_i\cdot {}^t(x^i_0, x^i_1, \cdots, x^i_n) = 0 \]
gives equations of $n+1$ hypersurfaces $D_j$, where ${}^t$ denotes the transpose as usual. We define the determinantal hypersurface
\[ W_i:= (\det(A_i) = 0) \subset Q_i = \PP^n_1 \times \cdots \times \widehat{\PP^n_i} \times \cdots \times \PP^n_N. \]
Note that the singularities of $W_i$ correspond to the points satisfying $\rank A_i \leq n-1$. Similarly, for the complete intersections $X_i\,(1\leq i \leq N)$, we can write $n+1$ hypersurfaces $D^i_{m_i}$ as the zero locus of
\[ \sum_{0\leq m_0 \leq n} \left(\sum_{\substack{(m_1, \dots, m_{i-1}, \\ m_{i+1}, \dots, m_N)}} b^{m_0 m_1\cdots m_{i-1} m_i m_{i+1} m_N} x^1_{m_1} \cdots x^{i-1}_{m_{i-1}} x^{i+1}_{m_{i+1}} \cdots x^N_{m_N} \right) x^0_{m_0}. \]
If we set the matrix
\[ A^{\prime}_i:= A^{\prime}_i(\mathbf{x}^1, \dots, \mathbf{x}^{i-1}, \mathbf{x}^{i+1}, \dots, \mathbf{x}^N):= \]
\[ \left(\sum b^{m_0 m_1\cdots m_{i-1} m_i m_{i+1} m_N} x^1_{m_1} \cdots x^{i-1}_{m_{i-1}} x^{i+1}_{m_{i+1}} \cdots x^N_{m_N} \right)_{m_i, m_0}, \]
then
\[ A^{\prime}_i \cdot {}^t(x^0_0, x^0_1, \dots, x^0_n)= 0 \]
gives equations of $n+1$ hypersurfaces $D^i_{m_i}$. We define the determinantal hypersurface
\[ W^{\prime}_i:= (\det(A^{\prime}_i) = 0) \subset Q_i = \PP^n_1 \times \cdots \times \widehat{\PP^n_i} \times \cdots \times \PP^n_N. \]
Since by definition
\[ b^{m_0 m_1 m_2 \dots m_N} = a_{m_0}^{m_1 m_2 \dots m_N}, \]
we know that indeed
\[ A^{\prime}_i = {}^tA_i, \]
and hence
\[ W^{\prime}_i = W_i \]
for each $i \in \{1, 2, \dots, N\}$. Moreover, we have natural projections $\alpha_i: X \to W_i$ and $\beta_i: X_i \to W_i$, and following diagram
\[ \xymatrix{
  X \ar[dr]_{\alpha_i}
  &  &    X_i \ar[dl]^{\beta_i}    \\
  & W_i  }.
\]

Let us remark that $W_i$ is a hypersurface of multidegree $(n+1, \cdots, n+1)$ in the space $Q_i = \PP^n_1 \times \cdots \times \widehat{\PP^n_i} \times \cdots \times \PP^n_N$ for each $i\in \{1, \dots, N\}$. So
\[ \dim W_i = n(N-1) - 1 = \dim X = \dim X_i, \]
and moreover,
\[ \rho(W_i) = N - 1 = \rho(X) - 1 = \rho(X_i) - 1 \] 
by the Noether--Lefschetz theorem for hypersurfaces, or equivalently, the relative Picard number $\rho(X/W_i) = \rho(X_i/W_i) = 1$.

The projection $\alpha_i: X \to W_i$ is the restriction of $\pi_i: P \to Q_i$ on $X$. Each fiber of $\alpha_i$ is a certain projective space (including a single point), so $\alpha_i$ admits connected fibers. Moreover, $\alpha_i: X \to W_i$ is isomorphic outside the proper subset of $W_i$ which consists of points satisfying $\rank A_i \leq n-1$, so $\alpha_i$ is birational. This implies that $\alpha_i: X \to W_i$ is the contraction corresponding to the codimension one face $\sum_{r\neq i}H_r$ under the natural isomorphism $\Nef(X) \cong \Nef(P) = \Cone(H_1, \dots, H_N)$. Similar assertion holds for $\beta_i: X_i \to W_i$.

For each pair of $i \neq j \in \{0, 1, \dots, N \}$, we can define the determinantal hypersurface $W_{ji}$ with similar properties discussed as above and the following diagram of birational contractions
\[ \xymatrix{
  X_j \ar[dr]_{}
  &  &    X_i \ar[dl]^{}    \\
  & W_{ji}   }.
\]
Here $X_0 = X$ and $W_{0i} = W_i$ for each $i \in \{1, \dots, N \}$.
%Again we can choose $X$ to be general \textit{a priori}, such that, besides Assumption \ref{assump},

%\begin{assumption}\label{assump2} These $W_{ij}$ are not linearly isomorphic mutually.                                 \end{assumption}

%Here a \textit{linear isomorphism} means a isomorphism induced by an automorphism of $P(N-1)$. Indeed, the dimension of              \[ \Aut(P(N-1)) = \mathrm{PGL}(n+1, \CC)^{N-1}\rtimes \mathfrak{S}_{N-1} \]                                              is equal to $(n^2 + 2n)\cdot(N-1)$, while we can choose $X$ in a  parameter space of dimension $(n+1)^2\cdot N$.

%In what follows, we will work under Assumptions \ref{assump} and \ref{assump2}.

\begin{lemma}\label{alter} Two birational contractions $\alpha_i: X \to W_i$ and $\beta_i: X_i \to W_i$ are either simultaneously divisorial or simultaneously small.
\end{lemma}

\begin{proof} The exceptional locus of $\alpha_i$ (resp. $\beta_i$) is the preimage of the subset of $W_i$ which consists of points satisfying $\rank A_i \leq n-1$ (resp. $\rank A^{\prime}_i \leq n-1$). The assertion follows from the fact that $A_i$ is the transpose of $A^{\prime}_i$.
\end{proof}

When $\alpha_i: X \to W_i$ and $\beta_i: X_i \to W_i$ are simultaneously small, we obtain the induced isomorphism in codimension one $\phi_i = \beta_i^{-1}\circ\alpha_i: X \dashrightarrow X_i$ which fits to the following commutative diagram
\[ \xymatrix{
  X \ar@{-->}[rr]^{\phi_i} \ar[dr]_{\alpha_i}
  &  &    X_i \ar[dl]^{\beta_i}    \\
  & W_i  }.
\]
Similarly, for each pair of $i \neq j \in \{0, 1, \dots, N \}$, we have the induced birational map $\phi_{ji}$:
\[ \xymatrix{
  X_j \ar@{-->}[rr]^{\phi_{ji}} \ar[dr]_{}
  &  &    X_i \ar[dl]^{}    \\
  & W_{ji}  },
\]
where $X_0 = X$ and $\phi_{0i} = \phi_i$ for each $i \in \{1, \dots, N \}$. Again we can choose $X$ general \textit{a priori} such that  besides Assumption \ref{assump},

\begin{assumption}\label{assump2} For each pair of $i \neq j \in \{0, 1, \dots, N \}$, the birational map $\phi_{ji}$ is not an isomorphism over $W_{ji}$.
\end{assumption}

Indeed, suppose that $\phi_{ji}$ is an isomorphism over $W_{ji}$ for some pair $i \neq j$. Take $k \in \{0, 1, \dots, N \} \setminus \{i, j\}$. Such $k$ always exists because $N \geq 2$. Consider the following commutative diagram:
\[ \xymatrix{
  X_k \ar@{-->}[rr]^{\phi_{kj}} \ar[dr]_{}
  &  &    X_j \ar[dl]^{} \ar[rr]^{\phi_{ji}} \ar[dr]_{} &  &    X_i \ar[dl]^{} \ar[dr]_{} \ar@{-->}[rr]^{\phi_{ik}}
  &  &    X_k  \ar[dl]^{}  \\
  & W_{kj} \ar@{-->}[rr]^{} & & W_{ji} \ar@{-->}[rr]^{} & & W_{ik}       },
\]
where $W_{kj} \dashrightarrow W_{ji}$ and $W_{ji} \dashrightarrow W_{ik}$ are naturally induced birational maps. Their composite $W_{kj} \to W_{ik}$ is an isomorphism (because $\phi_{ji}$ is an isomorphism). The induced birational map $W_{kj} \dashrightarrow X_j \to W_{ji}$ is given as follows: for a point
\[ \mathbf{a} := (\mathbf{a}^0, \cdots, \widehat{\mathbf{a}^k}, \cdots, \widehat{\mathbf{a}^j}, \cdots, \mathbf{a}^N) \in W_{kj}
\]
satisfying $\mathrm{rank}\,A_{kj}(\mathbf{a}) = n$, we have
\[ \mathbf{a} \mapsto (\mathbf{a}^0, \cdots, s(\mathbf{a}), \cdots, \widehat{\mathbf{a}^j}, \cdots, \mathbf{a}^N) \mapsto (\mathbf{a}^0, \cdots, \widehat{\mathbf{a}^i}, \cdots, \widehat{\mathbf{a}^j}, \cdots, \mathbf{a}^N),
\]
where $s(\mathbf{a}) \in \PP^n_k$ denotes the point corresponding to the solution of the system of linear equations given by $A_{kj}(\mathbf{a})$. The induced birational map $W_{ji} \dashrightarrow X_i \to W_{ik}$ is similar. By this way, we obtain the isomorphism $W_{kj} \to W_{ik}$, which further induces a specific $\CC$-algebra isomorphism $\det(A_{kj}) \to \det(A_{ik})$. This gives a proper closed subvariety of the parameter space of $X$. Therefore, we may choose $X$ outside these finitely many subvarieties indexed on pairs $i \neq j$ such that Assumption \ref{assump2} is satisfied.

Here and hereafter, we will work under Assumptions \ref{assump} and \ref{assump2}.

\begin{lemma}\label{flopx} When $\alpha_i: X \to W_i$ and $\beta_i: X_i \to W_i$ are simultaneously small, the map $\phi_i: X \dashrightarrow X_i$ is the flop of $\alpha: X \to W_i$.
\end{lemma}

\begin{proof} We note that $\rho(X/W_i) = \rho(X_i/W_i) = 1$. Take a divisor $H$ on $X$ which is relatively ample for $\alpha_i: X \to W_i$, and consider its strict transform $H^{\prime}$ under $\phi_i$. If $H^{\prime}$ is relative trivial for $\beta_i$, then $\rho(X_i) = \rho(W_i)$, which is a contradiction. Hence $H^{\prime}$ is either relatively ample or relatively anti-ample for $\beta_i$. If $H^{\prime}$ is relatively ample, then $\phi_i$ is an isomorphism over $W_i$. This contradicts Assumption \ref{assump2}. Therefore, the divisor $H^{\prime}$ is relatively anti-ample for $\beta_i$, and the map $\phi_i: X \dashrightarrow X_i$ is a flop by definition.
\end{proof}

\begin{corollary}\label{flops} For each $\ell \in \{0, 1, \dots, N\}$, all flops of $X_{\ell}$ are among these $X_{\ell}\dashrightarrow X_i$, where $i \in \{0, \dots, \dots, N\} \setminus \{\ell\}$.
\end{corollary}

\begin{proof} When $\ell = 0$, that is, $X_0 = X$, this is done in Lemma \ref{flopx}. By the symmetry in Construction \ref{constr} (under Assumptions \ref{assump} and \ref{assump2}) and Lemma \ref{alter}, the same holds for each $\ell$.
\end{proof}

Now we are ready to prove the first part of Theorem \ref{cipp}  concerning the finiteness of minimal models.

\begin{corollary}\label{maincor1} There are only finitely many minimal models of $X$ up to isomorphism $($as abstract varieties$)$.
\end{corollary}

\begin{proof} By a result of Kawamata (\cite[Theorem 1]{Ka08}), two birational Calabi--Yau varieties are connected by a sequence of flops up to isomorphism. By Proposition \ref{flops}, $X$ has at most $N+1$ minimal models $X_0 = X$ itself, $X_1, \dots, X_N$ up to isomorphism.
\end{proof}

%\begin{remark} $\mathrm{(1)}$ In view of Construction \ref{constr}, we indeed give an algorithm to find all the minimal models of $X$ up to isomorphism.                                              \medskip $\mathrm{(2)}$ Among these $X_i\, (i = 0, \dots, N)$, some of them might be isomorphic (as abstract varieties).     \end{remark}

\begin{corollary}\label{maincor2} There exists a rational polyhedral fundamental domain for the action of $\Bir(X)$ on $\Move(X) = \Movp(X)$.
\end{corollary}

\begin{proof} By Corollary \ref{maincor1}, $X$ admits only finitely many minimal model up to isomorphism (as abstract varieties), each of which is a smooth complete intersection of  hypersurfaces of multidegree $(1, \dots, 1)$ in the product space $P$. Therefore, the nef cone of each minimal model is isomorphic to the rational polyhedral cone $\Nef(P)$. The assertion then follows from Proposition \ref{mz4.1} and \ref{techcoh}.
\end{proof}

Now the proof of Theorem \ref{cipp} is completed. As a direct corollary, we obtain the finite presentation of the birational automorphism group of $X$.

\begin{corollary} The birational automorphism group $\Bir(X)$ of $X$ is finitely presented.
\end{corollary}

\begin{proof} By \cite[Proposition 2.4]{Og14}, the kernel of the map
\[ \Bir(X) \to \GL(N^1(X)_{\RR}, \RR) \]
is finite. On the other hand, by Theorem \ref{cipp} and \cite[Corollary 4.15]{Lo14}, the image of $\Bir(X) \to \GL(N^1(X)_{\RR}, \RR)$ is finitely presented, so the assertion follows.
\end{proof}

\section*{Acknowledgements}

I thank Professor Keiji Oguiso for his suggestions and encouragement. I also thank Professors Hiromichi Takagi, Chen Jiang and Hsueh-Yung Lin for answering my questions, and the referees for their comments. This work is partially supported by JSPS KAKENHI Grant Number 21J10242.

\end{document}